\documentclass[10pt]{amsart}
\usepackage[paperwidth=7in, paperheight=10in, margin=.875in]{geometry}
\usepackage{amsmath}
\usepackage{amsfonts,amsmath}
\usepackage{paralist}
\usepackage[colorlinks=true]{hyperref}
\hypersetup{urlcolor=blue, citecolor=red}
 \numberwithin{equation}{section}
\renewcommand{\theequation}{\arabic{section}.\arabic{equation}}

\def\open#1{\setbox0=\hbox{$#1$}
	\baselineskip = 0pt \vbox{\hbox{\hspace*{0.4 \wd0}\tiny
			$\circ$}\hbox{$#1$}} \baselineskip = 10pt\!}

\title[a fourth-order nonlinear parabolic equation
] 
{Local and global existence of solutions to a fourth-order parabolic equation modeling kinetic roughening and coarsening in thin films}

\author[Xiangsheng Xu]{}

\keywords{Biharmonic heat kernel;  interpolation inequality; local and global existence of weak solutions; nonlinear fourth order parabolic equations; thin film growth. To appear in Commun. Math. Sci..}

\email{xxu@math.msstate.edu}

\newtheorem{thm}{Theorem}[section]
\newtheorem{prop}[thm]{Proposition}
\newtheorem{lem}[thm]{Lemma}

\newtheorem{defn}[thm]{Definition}

\newtheorem{clm}[thm]{Claim}

\newtheorem{rem}[thm]{Remark}

\newcommand{\sk}{\mbox{$\psi_k $}}

\newcommand{\skx}{\mbox{$\psi_k (x)$}}

\newcommand{\uk}{\mbox{$u_k $}}

\newcommand{\uko}{\mbox{$u_{k-1} $}}

\newcommand{\utj}{\mbox{$\tilde{u}_j $}}

\newcommand{\ubj}{\mbox{$\overline{u}_j $}}
\newcommand{\sbj}{\mbox{$\overline{\psi}_j $}}
\newcommand{\rbj}{\mbox{$\overline{u}_j $}}

\newcommand{\omt}{\mbox{$\Omega_T$}}
\newcommand{\R}{\mathbb{R}}
\newcommand{\ds}{\Delta^2}
\newcommand{\ot}{\Omega_T}
\newcommand{\RN}{\mathbb{R}^N}
\newcommand{\f}{{\bf g}}
\newcommand{\ft}{{\bf g}}
\newcommand{\io}{{\int_{\Omega}}}
\newcommand{\iot}{{\int_{\Omega_T}}}
\newcommand{\gu}{\nabla u}
\newcommand{\du}{\Delta u}
\newcommand{\po}{\partial\Omega}
\newcommand{\rnt}{\RN\times(0,T)}
\newcommand{\dus}{\Delta^2 u}
\newcommand{\g}{{\bf g}}
\newcommand{\h}{{\bf h}}
\newcommand{\fn}{f_N}

\newcommand{\irn}{{\int_{\RN}}}
\newcommand{\irnt}{{\int_{\rnt}}}
\newcommand{\wk}{w_k}
\newcommand{\wko}{w_{k-1}}

\begin{document}
	\maketitle
	
	\centerline{\scshape Xiangsheng Xu}
	\medskip
	{\footnotesize
		\centerline{Department of Mathematics \& Statistics}
		\centerline{Mississippi State University}
		\centerline{ Mississippi State, MS 39762, USA}
	} 

	\bigskip
	

	\begin{abstract}In this paper we study both the Cauchy problem and the initial boundary value problem for the equation $\partial_tu+\mbox{div}\left(\nabla\Delta u-{\bf g}(\nabla u)\right)=0$. This equation  has been proposed as a continuum model for kinetic roughening and coarsening in thin films. In the Cauchy problem, we obtain that local existence of a weak solution is guaranteed as long as the vector-valued function ${\bf g}$ is continuous and the initial datum $u_0$ lies in $C^1(\mathbb{R}^N)$  with $\nabla u_0(x)$ being uniformly continuous and bounded on $\mathbb{R}^N$ and that the global existence assertion also holds true if we assume that ${\bf g}$ is locally Lipschitz and satisfies the growth condition $|{\bf g}(\xi) |\leq c|\xi|^\alpha$ for some $c>0,  \alpha\in (2, 3)$, $\sup_{\mathbb{R}^N}|\nabla u_0|<\infty$,  and the norm of $u_0$ in the space $L^{\frac{(\alpha-1)N}{3-\alpha}}(\mathbb{R}^N) $ is sufficiently small. This is done by exploring various properties of the biharmonic heat kernel. In the initial boundary value problem, we  assume that  ${\bf g}$ is continuous and satisfies the growth condition $|{\bf g}(\xi) |\leq c|\xi|^\alpha+c$ for some $c, \alpha\in (0,\infty)$. Our investigations reveal that if $\alpha\leq 1$ we have global existence of a weak solution, while if $1<\alpha<\frac{N^2+2N+4}{N^2}$	only a local existence theorem can be established. Our method here is based upon a new interpolation inequality, which may be of interest in its own right.

	\end{abstract}
	\section{Introduction.}\label{sec1}
In this paper we first study the initial boundary value problem for the equation mentioned in the abstract. Let $\Omega$ be a bounded domain in $\RN$ with boundary $\partial\Omega$. 
Then the problem can be formulated as follows:
\begin{eqnarray}
\partial_tu+\textup{div}\left(\nabla\Delta u-\f(\nabla u)\right)&=&0\ \ \textup{in $\ot\equiv\Omega\times(0,T)$,}\label{mp1}	\\
\nabla u\cdot\nu&=&0 \ \ \textup{on $\Sigma_T\equiv\partial\Omega\times(0,T)$,}\label{mp2}\\
\left(\nabla\Delta u-\f(\nabla u)\right)\cdot\nu&=&0 \ \ \textup{on $\Sigma_T$,}\label{mp3}\\
u(x,0)&=&u_0(x)\ \ \textup{on $\Omega$,}\label{mp4}
\end{eqnarray}
where $T>0$,  $\nu$ is the unit outward normal to $\partial\Omega$, and $\f$ is a vector-valued function from $\RN$
to $\RN$. The problem has been proposed by Ortiz et al \cite{ORS} as a continuum model for epitaxial thin film growth. In this context, $u$ is the scaled film height. The term $\Delta^2 u$ represents the capillarity-driven surface diffusion, while $\textup{div}\f(\gu)$ describes the upward hopping of atoms. The essential features of the growth of a thin film are depicted in \cite{ORS} (see also \cite{SK}). There are several stages. First, the mean deviation
of the film profile increases over time, known as roughing. After some initial period, islands form ``explosively'' in the substrate, where the number of islands decreases over time, whereas their size increases, a process called coarsening. Finally,
in the long run the film profile tends to exhibit a constant slope \cite{KOM}. 

The model here appears to be very successful in simulating the experimental observation \cite{KOM}. Analytical validation of the model has been carried out in \cite{FZ, KOM, SMG} and references therein under various assumptions on $\f$. Several function forms have been suggested for $\f$. For example, one can take 
\begin{equation*}
\f(\xi)=\left(c|\xi|^2+1\right)\xi.
\end{equation*} 
Here and in what follows the letter $c$ denotes a positive number whose value is either given or can be
theoretically derived from the given data. We refer the reader to \cite{KOM} for a discussion on how this type of functions can arise from applications. From the point of view of mathematical analysis, we see that functions of this kind have potentials, i.e., there are scalar functions $\phi$ such that 
\begin{equation}\label{cons}
\f=\nabla\phi.
\end{equation} Then \eqref{mp1} can be represented as the gradient
flow of the functional
\begin{equation}
E(u)=\io\left(\frac{1}{2}\left(\du\right)^2-\phi(\gu)\right)dx.
\end{equation}
That is,
\begin{equation}
\partial_tu=-\frac{\delta E}{\delta u}=-\textup{div}\left(\nabla\Delta u-\f(\nabla u)\right).
\end{equation}
This fact was employed in \cite{KOM, FZ} to analyze \eqref{mp1}-\eqref{mp4}. The case where $\f$ is not a conservative vector field
has been considered in \cite{SMG}. The method used there is to transform \eqref{mp1}-\eqref{mp4}
to an integral equation, and a solution is then established via successive approximations. The precise results from \cite{SMG} are stated in the following proposition.
\begin{prop}[A.N. Sandjo, S. Moutari, and Y. Gningue]\label{sand}Assume that  $\f$ is smooth with $\f(0)=0$ and $\nabla \f(0)=0$ and satisfies the condition
	\begin{equation*}
	|\nabla\f(\xi_1)-\nabla\f(\xi_2)|\leq c(|\xi_1|^{\alpha-2}+|\xi_2|^{\alpha-2})|\xi_1-\xi_2|\ \ \textup{for some $\alpha>2$ and $c>0$.}
	\end{equation*}
	If $\partial\Omega$ is $C^4$, $2<\alpha<3$, and $u_0\in L^p(\Omega)$, where $p=\frac{N(\alpha-1)}{3-\alpha}$, then the
	following statements hold true:
	\begin{enumerate}
		\item[\textup{(i)}]There is a $T>0$ such that \eqref{mp1}-\eqref{mp4} has a unique classical solution
		on $\ot$. Furthermore, for each $\delta\in(0,1)$ there holds
		\begin{equation}
		\sup_{0\leq t\leq T}\left\{t^{\frac{1-\delta}{2(\alpha-1)}}\|\nabla u\|_{\frac{N(\alpha-1)}{2\delta}}+t^{\frac{1}{2}}\|\nabla^2u\|_p\right\}<\infty,
		\end{equation}
		where $\|\cdot\|_p$ denotes the norm in $L^p(\Omega)$;
		\item[\textup{(ii)}] If $\|u_0\|_p$ is sufficiently small, then $u$ exists for $T=\infty$. 	
	\end{enumerate}
\end{prop}
Note that whenever $\f(\xi)=\varphi(\xi)\xi$, where $\varphi(\xi)$ is
a scalar function, the boundary condition \eqref{mp3} is equivalent to
\begin{equation}
\nabla\du\cdot\nu=0\ \ \textup{on $\Sigma_T$.}
\end{equation}

The objective of this paper is to present a new analytical perspective of
the model. Indeed, our assumptions are much weaker than those in the aforementioned papers. Our approach to \eqref{mp1}-\eqref{mp4} is based on  a new interpolation inequality and a Gronwall type argument.
It is worth poiting out
that our method allows for possible extension to the case where
the high order term in \eqref{mp1} is nonlinear \cite{LX2}.
We will pursue this possibility in a future study.

Precisely, we only assume:
\begin{enumerate}
	\item[(H1)] $\f$ is continuous and satisfies the growth condition
	\begin{equation}
	|\f(\xi)|\leq c|\xi|^\alpha+c\ \ \textup{ for some $c,\ \alpha \in (0,\infty)$};
	\end{equation}
	\item[(H2)] $\Omega$ is a bounded domain in $\RN$ with $C^{2,\gamma}$ boundary for some $\gamma\in (0,1)$;
	\item[(H3)] 
	$u_0\in W^{1,2}(\Omega)$.
\end{enumerate}


Since $\f$ is only continuous,  we are forced to seek a weak solution.
\begin{defn}
	We say that a function $u$ is a weak solution of \eqref{mp1}-\eqref{mp4} if:
	\begin{enumerate}
		\item[\textup{(D1)}] $u\in C([0,T]; L^2(\Omega))\cap L^2(0,T; W^{2,2}(\Omega)), \du\in L^2(0,T; W^{1,2}(\Omega)),
		|\f(\gu)|\in L^2(\ot)$;
		\item[\textup{(D2)}] $\gu\cdot\nu=0$ on $\Sigma_T$ and $u(x,0)=u_0(x)$ in $C([0,T]; L^2(\Omega))$;
		\item[\textup{(D3)}] for each smooth test function $\eta$ we have
		\begin{eqnarray}
		\lefteqn{\io u(x,T)\eta(x,T)dx-\io u_0(x)\eta(x,0)dx}\nonumber\\
		&&	-\iot u\partial_t \eta dxdt-\iot\left(\nabla\du-\f(\gu)\right)\nabla\eta dxdt=0.
		\end{eqnarray}
	\end{enumerate}
\end{defn}
Note that our integrability assumptions in (D1) are sufficient to guarantee that each term in (D2) and (D3) makes
sense. We can now state our main results.
\begin{thm}\label{mthm}
	Let (H1)-(H3) be satisfied. The following statements hold true:
	\begin{enumerate}
		\item[\textup{(R1)}] If $\alpha\leq 1$ then for each $T>0$ there is a weak solution $u$ to \eqref{mp1}-\eqref{mp4}
		on $\ot$;
		\item[\textup{(R2)}] Assume either
		\begin{equation}\label{con1}
		1<\alpha\leq \frac{N^2+2N+4}{N^2} \ \ \textup{and $N>2$ }
		\end{equation} 
		or \begin{equation}\label{con2}
		1<\alpha<3 \ \ \textup{and $N=2$. }
		\end{equation} Then there is a positive number $T_0$ determined
		by the given data such that \eqref{mp1}-\eqref{mp4} has a weak solution on $\Omega_{T_0}$.
	\end{enumerate}
\end{thm}
The proof of this theorem will be based on a new interpolation inequality, which we will introduce in Section \ref{sec3}. As we shall see,
the linear structure of the high order term in \eqref{mp1} is not essential in our proof , while the approach adopted in \cite{SMG} completely depends on this. 
We believe that our method can be extended to more general cases.

Next, we consider the Cauchy problem
\begin{eqnarray}
\partial_tu+\dus&=&\textup{div}\f(\nabla u)\ \ \textup{in $\RN\times(0,T)$,}\label{gp1}	\\
u(x,0)&=&u_0(x)\ \ \textup{on $\RN$.}\label{gp2}
\end{eqnarray}

A weak solution to \eqref{gp1}- \eqref{gp2} is defined as follow:
\begin{defn}\label{def2}
	We say that a function $u$ is a weak solution of \eqref{gp1}-\eqref{gp2} if:
	\begin{enumerate}
		\item[\textup{(D4)}] $u\in C([0,T]; L^2_{\textup{loc}}(\RN))\cap L^2(0,T; W^{2,2}_{\textup{loc}}(\RN)), \du\in L^2(0,T; W^{1,2}_{\textup{loc}}(\RN)),
		|\f(\gu)|\in L^2((0,T); L^2_{\textup{loc}}(\RN))$;
		\item[\textup{(D5)}] 
		$u(x,0)=u_0(x)$ in $C([0,T]; L^2_{\textup{loc}}(\RN))$;
		\item[\textup{(D6)}] for each smooth test function $\eta$ with compact support in the space variables we have
		\begin{eqnarray}
		\lefteqn{\irn u(x,T)\eta(x,T)dx-\irn u_0(x)\eta(x,0)dx}\nonumber\\
		&&	-\irnt u\partial_t \eta dxdt-\irnt\left(\nabla\du-\f(\gu)\right)\nabla\eta dxdt=0.
		\end{eqnarray}
	\end{enumerate}
\end{defn}
Our results are summarized in the following theorem.
\begin{thm}\label{mthm2}
	Assume that $u_0\in C^1(\RN)$
	and $\f$ is continuous. The following statements hold true:
	\begin{enumerate}
		\item[\textup{(R3)}] If $\nabla u_0$ is uniformly continuous and bounded on $\RN$, there is a positive number $T$ determined by the given data
		such that \eqref{gp1}-\eqref{gp2} has a weak solution $u$ on $\RN\times(0,T)$;
		\item[\textup{(R4)}] If $\f$ is locally Lipschitz and satisfies the growth condition
		\begin{equation}
		|\f(\xi)|\leq c|\xi|^\alpha\  \ \textup{for some $c>0$ and $3>\alpha>2$,}
		\end{equation}
		$\|\nabla u_0\|_\infty<\infty$, where $\|\cdot\|_p$ denotes the norm in $L^p(\RN)$,	and $\|u_0\|_{\frac{(\alpha-1)N}{3-\alpha}}$
		is sufficiently small, then  there is a  weak solution $u$ to 
		\eqref{gp1}-\eqref{gp2} for  $T=\infty$. Furthermore, there holds
		\begin{equation}
		\sup_{\RN\times [0,\infty)}|t^{\frac{1}{2(\alpha-1)}}\nabla u(x,t)|<\infty.
		\end{equation}
		Of course, in this case Definition \ref{def2} will have to be modified in an obvious way.
	\end{enumerate}
\end{thm}	
Note that if $u_0\in C^1(\RN)$ with $\lim_{|x|\rightarrow\infty}|\nabla u_0(x)|=0 $ then $\nabla u_0$ is uniformly continuous on $\RN$. This theorem says that local existence does not depend on the growth condition on $\f$. We suspect that the same is true in the situation considered in \cite{SMG}. That is, (i) in Proposition \ref{sand} can be improved if $u_0$ is smooth. We shall leave this to the interested reader. As for (ii) in the proposition, the difference is that
our assumptions on $\f$ are much weaker. This is partly due to the fact that one has classical solutions in (ii).
Nonetheless, Theorem \ref{mthm2} here is largely inspired by the results in \cite{SMG}.
We refer the reader to \cite{LW, TK}
for related work on the finite time blow-up of solutions to fourth-order nonlinear parabolic equations. See \cite{CM, GP} and the reference therein for the case where the
lower order term $\textup{div}\f(\gu)$ in \eqref{mp1} is replaced by $|u|^p$ for some $p>1$.

The difficulty in analyzing a fourth-order parabolic partial differential equation is due to the failure of the maximum principle. The technique of a majorizing kernel developed in \cite{GP} does not seem to be applicable to our case.  Thus we have not been able to develop a theory similar to the critical Fujita exponent \cite{GP} for our problem.

The rest of the paper is organized as follows. In Section \ref{sec2}, we collect some preparatory lemmas. Section \ref{sec3} is devoted to the
derivation of a priori estimates for classical solutions of \eqref{mp1}-\eqref{mp4}. This critically depends on the interpolation inequality we have established. In Section \ref{sec4}, we design an approximation scheme for \eqref{mp1}-\eqref{mp4}. The idea is to suitably transform our fourth-order equation into a system of two second-order equations. The advantage of doing this is that approximate solutions generated are very smooth. As a result, calculations performed in Section \ref{sec3} remain valid. The proof of Theorem \ref{mthm} is given in Section \ref{sec5}. We form a sequence of approximate solutions. Enough a priori estimates can be gathered for the sequence to justify passing to the limit. In the very last section we present the proof of Theorem \ref{mthm2}. Here we exploit certain properties of the biharmonic heat kernel.

\section{Preliminaries.}\label{sec2}
In this section we collect some useful classical results. The next lemma has played an important role in our approach.
\begin{lem}\label{gron}
	Let $y(t)$ be a non-negative, differentiable function on $[0,\infty)$ satisfying the differential inequality
	\begin{equation}
	y^\prime(t)\leq c_1y^{1+\sigma}(t)+c_2\ \ \ \textup{on $[0, \infty)$ for some $\sigma, c_1, c_2\in (0,\infty)$.}\label{le11}
	\end{equation}
	Then we have
	\begin{equation}
	y(t)\leq \frac{1}{\left[\left(\left(v_0^{-\sigma}+\frac{c_1}{c_2}\right)e^{-\sigma c_2 t}-\frac{c_1}{c_2}\right)^+\right]^{\frac{1}{\sigma}}}-1,
	\end{equation}
	where $v_0=y(0)+1$.
\end{lem}
\begin{proof}
	Set \begin{equation}
	v=y+1.
	\end{equation}
	Then we can easily derive from \eqref{le11} that
	\begin{equation}
	v^\prime\leq c_1v^{1+\sigma}+c_2v.
	\end{equation}
	Subtract $c_2v$ from both sides of the above inequality and divide through the resulting inequality by $v^{1+\sigma}$ to obtain
	\begin{equation}
	\left(v^{-\sigma}\right)^\prime+\sigma c_2v^{-\sigma}\geq -c_1\sigma.
	\end{equation}
	Multiply through the inequality by $e^{\sigma c_2 t}$ and integrate to get
	\begin{equation}
	e^{\sigma c_2 t}v^{-\sigma}-v_0^{-\sigma}\geq -\frac{c_1}{c_2}\left(e^{\sigma c_2 t}-1\right),
	\end{equation}
	from whence the lemma follows.
\end{proof}

Our  global existence  assertion in the case of small initial data relies on the following lemma.
\begin{lem}\label{small}
	Let $\alpha,\lambda\in (0,\infty)$ be given and $\{b_k\}$ a sequence of nonnegative numbers with the property
	\begin{equation}
	b_k\leq b_0+\lambda b_{k-1}^{1+\alpha}\ \ \textup{for $k=1,2,\cdots.$}
	\end{equation}
	If $2\lambda(2b_0)^\alpha<1$, then
	\begin{equation}
	b_k\leq \frac{b_0}{1-\lambda(2b_0)^\alpha}\ \ \textup{for all $k\geq 0$.}
	\end{equation}
\end{lem}
This lemma can easily be established via induction.

We next recall some relevant information on the biharmonic heat kernel. Set
\begin{equation}\label{fn}
\fn(\eta)=\eta^{1-N}\int_{0}^{\infty}e^{-s^4}(\eta s)^{\frac{N}{2}}J_{\frac{N-2}{2}}(\eta s)ds.
\end{equation}
Here $J_\nu$ denotes the $\nu$th Bessel function of the first kind.
Then the biharmonic heat kernel $b_N(x,t)$ has the expression 
\begin{equation}
b_N(x,t)=\alpha_Nt^{-\frac{N}{4}}\fn\left(\frac{|x|}{t^{\frac{1}{4}}}\right),
\end{equation}
where $\alpha_N>0$ is a constant to be specified. 
The following lemma summarizes some relevant properties of $\fn$.
\begin{lem}
	The function $\fn$ defined in \eqref{fn} has the following properties:
	\begin{enumerate}
		\item[\textup{(F1)}] $\fn$ is a solution of the ordinary differential equation
		\begin{equation}
		\fn^{\prime\prime\prime}(\eta)+\frac{N-1}{\eta}	\fn^{\prime\prime}(\eta)-\frac{N-1}{\eta^2}	\fn^{\prime}(\eta)-\frac{\eta}{4}\fn(\eta)=0;
		\end{equation}
		\item[\textup{(F2)}] There exist positive numbers $K=K(N), \mu=\mu(N)$ with
		\begin{equation}\label{gp}
		|\fn(\eta)|\leq K\textup{exp}(-\mu\eta^{
			\frac{4}{3}})\ \ \textup{for $\eta\geq 0$;}
		\end{equation} 
		\item[\textup{(F3)}]$\fn^\prime(\eta)=-\eta f_{N+2}(\eta)$;
		\item[\textup{(F4)}] $\fn(\eta)$ changes signs infinitely many times as $\eta\rightarrow\infty$, and  there holds
		\begin{equation}
		\int_{0}^{\infty}\eta^{N-1-\beta}\fn(\eta)d\eta >0\ \ \textup{for each $\beta\in [0, N)$.}
		\end{equation}
	\end{enumerate}
\end{lem}
The items (F1), (F3), and (F4) can be found in \cite{FGG}, while (F2) is taken from \cite{GP}.

Denote by $B(0,r)$ the ball centered at the origin with radius $r$. We calculate
\begin{eqnarray}
\irn\fn\left(\frac{|y|}{t^{\frac{1}{4}}}\right)dy
&=&\int_{0}^{\infty}\int_{\partial B(0,r)}\fn\left(\frac{|y|}{t^{\frac{1}{4}}}\right)d\mathcal{H}^{N-1}dr\nonumber\\
&=&\omega_N\int_{0}^{\infty}\fn\left(\frac{r}{t^{\frac{1}{4}}}\right)r^{N-1}dr\nonumber\\
&=&\omega_Nt^{\frac{N}{4}}\int_{0}^{\infty}\fn(\eta)\eta^{N-1}d\eta,\label{stan}
\end{eqnarray}
where $\omega_N$ is the volume of the unit ball in $\RN$. Thus we can choose $\alpha_N$ so that 
\begin{equation}\label{norm}
\alpha_Nt^{-\frac{N}{4}}\irn\fn\left(\frac{|x|}{t^{\frac{1}{4}}}\right)dx=1.
\end{equation}

An easy consequence of (F2) and (F3) is that for each non-negative integer $n$ there are two positive numbers $K=K(n,N), \mu=\mu(n, N)$ such that
\begin{equation}
\left|\fn^{(n)}(\eta)\right|\leq K\exp(-\mu\eta^{\frac{4}{3}})\ \ \textup{on $[0,\infty)$,}
\end{equation}
where $\fn^{(n)}$ denotes the nth order derivative of $\fn$. Subsequently, for each $q>1$ we have
\begin{eqnarray}
\irn\left|\fn^{(n)}\left(\frac{|y|}{t^{\frac{1}{4}}}\right)\right|^qdy
&=&\int_{0}^{\infty}\int_{\partial B(0,r)}\left|\fn^{(n)}\left(\frac{|y|}{t^{\frac{1}{4}}}\right)\right|^qd\mathcal{H}^{N-1}dr\nonumber\\
&=&\omega_N\int_{0}^{\infty}\left|\fn^{(n)}\left(\frac{r}{t^{\frac{1}{4}}}\right)\right|^qr^{N-1}dr\nonumber\\
&=&\omega_Nt^{\frac{N}{4}}\int_{0}^{\infty}\left|\fn^{(n)}(\eta)\right|^q\eta^{(N-1)}d\eta\nonumber\\
&\leq &ct^{\frac{N}{4}}\int_{0}^{\infty}\textup{exp}(-\mu q\eta^{
	\frac{4}{3}})\eta^{(N-1)}d\eta\nonumber\\
&\leq &ct^{\frac{N}{4}}.\label{stan2}
\end{eqnarray}

Remember that the convolution of  two functions $f$ and $g$ on $\RN$ is defined to be
\begin{equation}
f*g(x)=\irn f(x-y)g(y)dy.
\end{equation}
Young's inequality for convolutions is as follows: Let $h=f*g$, then
\begin{equation}\label{youn}
\|h\|_q\leq \|f\|_p\|g\|_r,
\end{equation} 
where $1\leq p, q, r\leq \infty$ and $\frac{1}{q}=\frac{1}{p}+\frac{1}{r}-1$. When $r$ is the conjugate index to $p$ (namely $\frac{1}{p}+\frac{1}{r}=1$), then $q=\infty$. It is also frequently used with $r=1$. Then $p=q$ and \eqref{youn} becomes
\begin{equation}\label{youn1}
\|h\|_q\leq \|f\|_q\|g\|_1.
\end{equation} 

Relevant interpolation inequalities for Sobolev spaces are listed in the following lemma.
\begin{lem}	Let (H2) be satisfied.
	Then we have:
	\begin{enumerate}
		\item There is a positive number c such that
		\begin{equation}
		\|\gu\|_{2^*}\leq c\left(\|\nabla^2 u\|_2+\|u\|_2\right)\ \ \textup{for all $u\in
			W^{2,2}(\Omega)$},
		\end{equation}
		where $2^*=\frac{2N}{N-2}$ if $N>2$ and any number bigger than $2$ if $N=2$. Obviously, here $\|\cdot\|_p$ denotes the norm in the space $L^p(\Omega)$;
		\item Let $p\in [2, 2^*)$. Then for each $\varepsilon >0$ there is a positive number $c=c(\varepsilon, p)$ such that
		\begin{equation}
		\|\nabla u\|_p\leq \varepsilon\|\nabla^2 u\|_2+c\|u\|_2\ \ \textup{for all $u\in
			W^{2,2}(\Omega)$}.	\label{otn10}
		\end{equation}
	\end{enumerate}
\end{lem}

\begin{lem}
	Let (H2) be satisfied and
	$u$ be
	a weak solution of the Neumann boundary
	value problem
	\begin{eqnarray}
	-\Delta u &=& g\ \ \textup{in $\Omega$,}\\
	\gu\cdot\nu&=&0 \ \ \textup{on $\po$.}
	\end{eqnarray}
	Then for each $p>1$ there is a positive constant $c$
	depending only on $N, p, \Omega$ and the smoothness of the boundary
	such that
	\begin{equation}\label{cz}
	\|u\|_{W^{2,p}(\Omega)}\equiv\|\nabla^2u\|_p+\|\gu\|_p+\|u\|_p\leq c \|g\|_p=c\|\du\|_p.
	\end{equation}
\end{lem}
This lemma is the classical Calder\'{o}n-Zygmund inequality. It can be found in \cite{G}.


\section{A priori estimates.}\label{sec3}
Our a priori estimates rely on an interpolation inequality, which may be of interest in its own right. To state the inequality, we 
set
\begin{equation*}
u_\Omega=\frac{1}{|\Omega|}\io u dx,
\end{equation*}
where $|\Omega|$ denotes the Lebesgue measure of $\Omega$. We have the well-known Sobolev inequality
\begin{equation}
\|u-u_\Omega\|_{2^*}\leq c_\Omega\|\gu\|_2\ \ \textup{for all $u\in W^{1,2}(\Omega)$}.\label{soi}
\end{equation} 
\begin{prop}\label{interp}Assume that $N>2$ and (H2) holds.
	Take $\alpha\in(1, \frac{N}{N-2})$. For $k=0, 1, 2, \cdots$ define the sequences
	\begin{eqnarray}
	a_k&=&\frac{2N}{N-2}-2\left(\frac{N}{N-2}-\alpha\right)\left(\frac{N}{2}\right)^k,\\
	b_k&=&\frac{1-\left(\frac{2}{N}\right)^k}{1-\frac{2}{N}}.
	\end{eqnarray} Then there holds
	\begin{equation}
	\io|\gu|^{2\alpha}dx\leq\left(((2\alpha-2)c+1)c_\Omega\right)^{b_k}\left(\|\du\|_{2^*}\right)^{b_k}\left(\|\gu\|_{2}\right)^{b_k}
	\left(\io|\gu|^{a_k}dx\right)^{\left(\frac{2}{N}\right)^k}\label{int}
	\end{equation}
	for $k\in \{0, 1, \cdots\}$ with $a_k\geq 0$ and all $u\in W^{2, 2^*}(\Omega)$ with $\gu\cdot\nu=$
	on $\po$, where $c$ is the same as the one in (\ref{cz}).
\end{prop}
\begin{proof} First we observe that $a_k$ is the solution of the difference equation
	\begin{equation}\label{int1}
	a_k=(a_{k-1}-2)\frac{N}{2}
	\end{equation}
	coupled with the initial condition 
	\begin{equation*}
	a_0=2\alpha.
	\end{equation*}
	Our assumption on $\alpha$ implies that $\{a_k\}$ is decreasing.  We use mathematical induction. Note that $b_0=0$. Thus if $k=0$, then  (\ref{int}) is trivially true. Assume that (\ref{int}) holds for some
	$k=n\geq 0$ and $a_{n+1}\geq 0$. In view of (\ref{int1}), we have $a_n\geq 2$.
	An elementary calculation shows that
	\begin{equation*}
	\textup{div}\left(|\gu|^{a_n-2}\gu\right)=(a_n-2)|\gu|^{a_n-4}(\nabla^2u\gu)\cdot\gu+|\gu|^{a_n-2}\du.
	\end{equation*}
	With this in mind, we estimate from \eqref{soi} and \eqref{cz} that
	\begin{eqnarray}
	\io|\gu|^{a_n}dx
	&=&\io|\gu|^{a_{n}-2}\gu\cdot\gu  dx\nonumber\\
	&=&\io|\gu|^{a_{n}-2}\gu\cdot\nabla (u-u_\Omega) dx\nonumber\\
	&=&-\io\textup{div}\left(|\gu|^{a_n-2}\gu\right) (u-u_\Omega)dx\nonumber\\
	&\leq&(a_n-2)\io|\gu|^{a_n-2}|\nabla^2u||u-u_\Omega|  dx+\io|\gu|^{a_n-2}|\du||u-u_\Omega|   dx\nonumber\\
	&\leq&\left((2\alpha -2)\|\nabla^2u\|_{2^*}+ \|\du\|_{2^*}\right)\left(\io|\gu|^{(a_n-2)\frac{2N}{N+2}}|u-u_\Omega|^{\frac{2N}{N+2}}  dx\right)^{\frac{N+2}{2N}}\nonumber\\
	&\leq&((2\alpha -2)c+1)\|\du\|_{2^*}\left(\io|\gu|^{(a_{n}-2)\frac{N}{2}}dx\right)^{\frac{2}{N}}\left(\io|u-u_\Omega|^{\frac{2N}{N-2}}  dx\right)^{\frac{N-2}{2N}}\nonumber\\
	&\leq &((2\alpha -2)c+1)c_\Omega\|\du\|_{2^*}\left(\io|\gu|^{2}  dx\right)^{\frac{1}{2}}\left(\io|\gu|^{a_{n+1}}dx\right)^{\frac{2}{N}}.
	\end{eqnarray}
	The last step is due to (\ref{soi}). By our assumption, (\ref{int}) is true for $k=n$. Therefore, we have
	\begin{eqnarray}
	\io|\gu|^{2\alpha}dx&\leq&\left(((2\alpha-2)c+1)c_\Omega\right)^{b_n}\left(\|\du\|_{2^*}\right)^{b_n}\left(\|\gu\|_{2}\right)^{b_n}	\left(\io|\gu|^{a_n}dx\right)^{\left(\frac{2}{N}\right)^n}\nonumber\\
	&\leq &   \left(((2\alpha-2)c+1)c_\Omega\right)^{b_n+\left(\frac{2}{N}\right)^n}\left(\|\du\|_{2^*}\right)^{b_n+\left(\frac{2}{N}\right)^n}\left(\|\gu\|_{2}\right)^{b_n+\left(\frac{2}{N}\right)^n}\nonumber\\
	&&\cdot	\left(\io|\gu|^{a_{n+1}}dx\right)^{\left(\frac{2}{N}\right)^{n+1}}\nonumber\\   
	&\leq&\left(((2\alpha-2)c+1)c_\Omega\right)^{b_{n+1}}\left(\|\du\|_{2^*}\right)^{b_{n+1}}\left(\|\gu\|_{2}\right)^{b_{n+1}}\nonumber\\
	&&\cdot\left(\io|\gu|^{a_{n+1}}dx\right)^{\left(\frac{2}{N}\right)^{n+1}} .                   
	\end{eqnarray}
	This completes the proof.
\end{proof}

In \cite{RS}, the authors established the following inequality
\begin{equation}
\int_{\mathbb{R}^N}\psi^{s+2}|\nabla u|^{s+2}dx\leq c\|u\|^2_{BMO}\left(\int_{\mathbb{R}^N}\psi^{s+2}|\nabla u|^{s-2}|\nabla^2u|^2dx+\|\nabla\psi\|_\infty^2\int_{\mathbb{R}^N}\psi^{s}|\nabla u|^{s}dx\right), 
\end{equation}
where $ s\geq 2$ and $\psi$ is a non-negative cut-off function in $C^\infty_0(\RN)$. This inequality is local in nature, and 
it does not seem to be applicable to our situation due to the Neumann
boundary conditions involved. Furthermore, in general, it is not easy to derive BMO bounds for solutions to a fourth-order nonlinear equation.
\begin{prop}\label{estim}Let (H1)-(H3) be satisfied and $u$ be a classical solution of (\ref{mp1})-(\ref{mp4}).
	The following statements hold true:
	\begin{enumerate}
		\item[\textup{(e1)}] If $\alpha\leq 1$ for each $T>0$ then there is a positive number $c$ determined only by the given data such that 
		\begin{equation}\label{e111}
		\max_{0\leq t\leq T}\io (u^2+|\gu|^2)dx+\int_{\Omega_{T}}\left(\Delta u\right)^2dxdt +\int_{\Omega_{T}}|\nabla\Delta u|^2dxdt\leq c;
		\end{equation}
		\item[\textup{(e2)}] Assume either \eqref{con1} or \eqref{con2}. Then
		there exists a positive number $T_0$ with
		\begin{eqnarray}
		\lefteqn{\max_{0\leq t\leq T_0}\io (u^2+|\gu|^2)dx+\int_{\Omega_{T_0}}\left(\Delta u\right)^2dxdt}\nonumber\\ &&+\int_{\Omega_{T_0}}|\nabla\Delta u|^2dxdt+\int_{\Omega_{T_0}}|\nabla u|^{2\alpha}dxdt\leq c.
		\end{eqnarray}
		Here  $c$ is also determined by the given data.
	\end{enumerate}
\end{prop}
\begin{proof} 
	Multiply each term in (\ref{mp1}) by $u$ and then integrate the resulting equation over $\Omega$ to derive
	\begin{eqnarray}
	\frac{1}{2}\frac{d}{dt}\io u^2dx+\io\left(\Delta u\right)^2dx&=&-\io\f(\nabla u)\cdot\nabla udx\nonumber\\
	&\leq& c\io|\nabla u|^{1+\alpha}dx+c.\label{ap2}
	\end{eqnarray}
	Multiplying through (\ref{mp1}) by $\du$ and
	then integrating over $\Omega$, after some calculations we arrive at 
	\begin{eqnarray}
	\frac{1}{2}\frac{d}{dt}\io |\gu|^2dx+\io|\nabla\du|^2dx&=&\io\f(\gu)\nabla\du dx\nonumber\\
	&\leq& \frac{1}{2}\io|\nabla\du|^2dx+ \frac{1}{2}\io|\f(\gu)|^2dx\nonumber\\
	&\leq& \frac{1}{2}\io|\nabla\du|^2dx+ c\io|\gu|^{2\alpha}dx+c.\label{ap4}
	\end{eqnarray}
	If $\alpha\leq 1$, we can apply Gronwall's inequality to obtain
	\begin{equation}
	\max_{0\leq t\leq T}\io |\gu|^2dx+\iot|\nabla\du|^2dxdt\leq c.\label{ap5}
	\end{equation}
	This together with (\ref{ap2}) yields
	\begin{equation}
	\max_{0\leq t\leq T}\io u^2dx+\iot\left(\Delta u\right)^2dxdt\leq c.\label{ap3}
	\end{equation}
	This completes the proof of (e1).
	
	From here on, we assume \eqref{con1}.
	We easily deduce from \eqref{ap2} and \eqref{ap4} that 
	\begin{equation}
	\frac{d}{dt}\io \left(u^2+|\gu|^2\right)dx+\io\left(\du\right)^2dx+\io|\nabla\du|^2dx\leq c\io|\gu|^{2\alpha}dx+c.\label{ap6}
	\end{equation}
	To apply Proposition \ref{interp}, we choose $k$ so that
	\begin{eqnarray}
	b_k&<&2,\label{bk}\\
	a_k&\leq &2.\label{ak}
	\end{eqnarray}
	Then for each $\varepsilon>0$ we estimate
	\begin{eqnarray}
	\io|\gu|^{2\alpha}dx&\leq&\left((2\alpha-2)c+1)c_\Omega\right)^{b_k}\left(\|\du\|_{2^*}\right)^{b_k}\left(\|\gu\|_{2}\right)^{b_k}
	\left(\io|\gu|^{a_k}dx\right)^{\left(\frac{2}{N}\right)^k}\nonumber\\
	&\leq &c\left(\|\nabla\du\|_2+\|\du\|_2\right)^{b_k}\left(\|\gu\|_{2}\right)^{b_k}
	\left(\io|\gu|^{a_k}dx\right)^{\left(\frac{2}{N}\right)^k}\nonumber\\
	&\leq &\varepsilon\left(\|\nabla\du\|_2^2+\|\du\|_2^2\right)+c(\varepsilon)\left(\|\gu\|_{2}\right)^{\frac{2b_k}{2-b_k}}
	\left(\io|\gu|^{a_k}dx\right)^{\left(\frac{2}{N}\right)^k\frac{2}{2-b_k}}\nonumber\\
	&\leq &\varepsilon\left(\|\nabla\du\|_2^2+\|\du\|_2^2\right)+c(\varepsilon)
	\left(\io|\gu|^{2}dx\right)^{\frac{b_k}{2-b_k}+\left(\frac{2}{N}\right)^k\frac{a_k}{2-b_k}}.\label{gue}
	\end{eqnarray}
	Use this in \eqref{ap6} and choose $\varepsilon$ suitably small to derive
	\begin{equation}\label{yes}
	y^\prime\leq cy^{1+\sigma}+c,
	\end{equation} 
	where \begin{eqnarray}
	y(t)&=&\io\left(u^2+|\gu|^2\right)dx,\\
	1+\sigma&=&\frac{b_k}{2-b_k}+\left(\frac{2}{N}\right)^k\frac{a_k}{2-b_k}.\label{bk2}
	\end{eqnarray}
	It turns out that the condition \eqref{con1} is to ensure that \eqref{ak} and \eqref{bk} hold for $k=2$. To see this,  we have
	\begin{eqnarray}
	b_2&=&1+\frac{2}{N}<2,\\
	a_2&=& \frac{2N}{N-2}-2\left(\frac{N}{N-2}-\alpha\right)\left(\frac{N}{2}\right)^2\nonumber\\
	&=&\frac{N^2}{2}\alpha-\frac{N(N+2)}{2}\leq 2.
	\end{eqnarray}
	Use $b_2$ in \eqref{bk2} to see that $\sigma>0$. Consequently, we can
	apply Lemma \ref{gron} to obtain
	\begin{equation}\label{yt1}
	y(t)\leq \frac{1}{\left[\left(\left(v_0^{-\sigma}+1\right)e^{-\sigma c t}-1\right)^+\right]^{\frac{1}{\sigma}}}-1,
	\end{equation}
	where
	\begin{equation}
	v_0=\io\left(u_0^2+|\nabla u_0|^2\right)dx+1.
	\end{equation}
	Note that the right-hand side of \eqref{yt1} blows up only when $t$ reaches $\frac{\ln(v_0^{-\sigma}+1)}{\sigma c}$.
	It follows that for each $T_0\in (0, \frac{\ln(v_0^{-\sigma}+1)}{\sigma c})$ the function $y(t)$ is bounded on
	$[0, T_0]$. Equipped with this, we can combine \eqref{ap6} with \eqref{gue} to obtain the desired result.
	
	Now we assume \eqref{con2} holds. In this case, for each $s>1$ there is a positive number $c=c(s,\Omega)$ such that
	\begin{eqnarray}
	\|u-u_\Omega\|_{\frac{s}{s-1}}&\leq& c(s,\Omega)\|\gu\|_2,\\
	\|\du\|_{\frac{s}{s-1}}&\leq& c(s,\Omega)\left(\|\nabla\du\|_2+\|\du\|_2\right).
	\end{eqnarray}
	We can easily infer from the proof of Proposition \ref{interp} that
	\begin{equation}
	\io|\gu|^{2\alpha}dx\leq\left(((2\alpha-2)c+1)c(s,\Omega)\right)^{b_k}\left(\|\du\|_{2^*}\right)^{b_k}\left(\|\gu\|_{2}\right)^{b_k}
	\left(\io|\gu|^{a_k}dx\right)^{\frac{1}{s^{2k}}},
	\end{equation} 
	where
	\begin{eqnarray}
	a_k&=&\frac{2s^2}{s^2-1}-2\left(\frac{s^2}{s^2-1}-\alpha\right)s^{2k},\\
	b_k&=&\frac{1-\frac{1}{s^{2k}}}{1-\frac{1}{s^2}},\ \ k=0, 1, 2, \cdots .
	\end{eqnarray} 
	Note that $b_2=1+\frac{1}{s^2}<2$. The condition $a_2\leq 2$ leads to
	\begin{equation}
	\alpha\leq \frac{s^4+s^2+1}{s^4}\uparrow 3\ \ \textup{as $s\rightarrow 1^+$.}
	\end{equation} Thus if $\alpha<3$ we can always find an $s>1$ such that 
	\begin{equation}
	\alpha\leq \frac{s^4+s^2+1}{s^4}<3.
	\end{equation}
	The rest is entirely similar to the argument after \eqref{yt1}.
\end{proof}	
\begin{rem}
	If $\f$ is conservative, i.e., \eqref{cons} holds, then we can use $\partial_tu$ as a test function in
	\eqref{mp1} and thereby derive
	\begin{equation}
	\io(\partial_tu)^2dx+\frac{1}{2}\frac{d}{dt}\io\left(\du\right)^2dx+\frac{d}{dt}\io\phi(\gu)dx=0.
	\end{equation}
	In particular, if $\f(\xi)=|\xi|^{\alpha-1}\xi$, $\alpha>0$, there is a positive number $c$ depending
	only on $u_0$ and $\alpha$ such that
	\begin{equation}
	\iot(\partial_tu)^2dxdt+\max_{0\leq t\leq T}\left(\io\left((\du)^2+|\gu|^{\alpha+1}\right)dx\right)\leq c\ \ \textup{for each $T>0$}.
	\end{equation}
	In this case, our analysis implies that \eqref{mp1}-\eqref{mp4} has a global weak solution for each $\alpha>0$, provided that we slightly modify the definition of a weak solution.
\end{rem}

\section{Approximate problems.}\label{sec4}
To design our approximate problems, we introduce a new unknown function $\psi$ so that
\begin{equation}
-\Delta u+\tau u=\psi,\label{exi0}
\end{equation}
where $1>\tau>0$. 
Obviously, the addition of the term $\tau u$ is due to the Neumann boundary condition (\ref{mp2}).
Note that we have $\psi=-\du$ as $\tau\rightarrow 0$. 
To approximate (\ref{mp1}), 
we employ an implicit discretization scheme for the time variable in (\ref{mp1}). This leads
to the consideration of
the following system of two second-order elliptic equations
\begin{eqnarray}
-\Delta {\psi}+\tau\psi -\textup{div}\ft(\gu)&=&- \frac{u-v}{\tau}\ \ \ \textup{in $\Omega$},\label{exi1}\\
-\Delta u+\tau u&=&\psi\ \ \ \textup{in $\Omega$}\label{exi2}
\end{eqnarray}
coupled with the boundary conditions
\begin{eqnarray}
\gu\cdot\nu&=& 0\ \ \ \textup{on $\partial\Omega$},\label{exi3}\\
\left(\nabla \psi+\ft(\gu)\right)\cdot\nu&=&0\ \ \ \textup{on $\partial\Omega$},\label{exi4}
\end{eqnarray}
where  $v$ is either the initial datum $u_0$ or the $u$ component of the solution $(u, \psi)$ obtained in a preceding step in the scheme. The number $\tau$ can be viewed as the step size.
The term $\tau\psi$ in (\ref{exi1}) is also needed due to the Neumann boundary condition.

The construction of our approximate solutions will be based upon this problem.

\begin{prop}\label{app}
	Let $\f(\xi)$ be a continuous function on $\RN$, and assume that (H2) holds and $v\in L^\infty(\Omega)$.
	Then there is a 
	weak solution $(\psi, u)$ to (\ref{exi1})-(\ref{exi4}) with $\psi\in W^{1,p}(\Omega), u\in W^{2,p}(\Omega)$ for
	each $p>1$. 
\end{prop}
\begin{proof}
	A solution will be constructed 
	via the Leray-Schauder Theorem
	(\cite{GT}, p. 280). To apply this theorem,
	we define an operator $B$ from $L^\infty(\Omega)$
	into itself as follows: for each $w\in L^\infty(\Omega)$ we say $B(w)=\psi$ if $\psi$ solves the equation
	\eqref{exi1} with the boundary condition \eqref{exi4},
	in which $u$ is the solution of \eqref{exi2} coupled with \eqref{exi3}.
	To see that $B$ is well-defined, we observe from the classical existence and regularity theory for linear elliptic equations 
	that (\ref{exi2})-(\ref{exi3}) has a unique weak solution $u$ which can be shown to lie in the space $W^{2,p}(\Omega)$ for 
	each $p>1$. Thus
	$u\in C^{1,\beta}(\overline{\Omega})$ for some $\beta\in(0, 1)$ and $\f(\gu)$ is a bounded, continuous function on $\overline{\Omega}$. 
	This enables us to conclude that (\ref{exi1}) combined with (\ref{exi4}) also has a unique solution $\psi$ in the space $W^{1,2}(\Omega)\cap C^\gamma(\overline{\Omega})$ for
	some $\gamma\in (0,1)$. 
	In view of these results, we can claim that $B$ is well-defined, continuous, and maps bounded sets into precompact ones. It remains to show that there is a positive number $c$ such that
	\begin{equation}
	\|\psi\|_\infty\leq c\label{ot8}
	\end{equation}
	for all $\psi\in L^\infty(\Omega)$ and $\sigma\in (0,1]$ satisfying
	$$\psi=\sigma B(\psi).$$
	This equation is equivalent to the boundary value problem
	\begin{eqnarray}
	-\textup{div}\left(\nabla\psi+\sigma\ft(\gu)\right)+\tau\psi  &=&-\sigma\frac{u-v}{\tau} \ \ \ \textup{in $\Omega$},\label{ot9}\\
	-\Delta u+ \tau u&=& \psi \ \ \ \textup{in $\Omega$},\label{ot10}\\
	\gu\cdot\nu&=& 0\ \ \ \textup{on $\partial\Omega$,}\\
	\left(\nabla\psi+\sigma\ft(\gu)\right)\cdot\nu&=& 0\ \ \ \textup{on $\partial\Omega$}.
	\end{eqnarray}
	Let $p>2$ be given. We easily check  that the function $|\psi|^{p-2} \psi$ lies in $W^{1,2}(\Omega)$ and $\nabla\left(|\psi|^{p-2} \psi\right)=(p-1)|\psi|^{p-2}\nabla\psi$. Multiply through (\ref{ot9}) by
	this function and integrate the resulting equation over $\Omega$ to obtain
	\begin{eqnarray*}
		(p-1)\int_\Omega|\psi|^{p-2}|\nabla\psi|^2\, dx
		+\tau\int_\Omega|\psi|^p\, dx&=&-\frac{\sigma}{\tau}\int_\Omega(u-v)|\psi|^{p-2} \psi \, dx\nonumber\\
		&&-\sigma(p-1)\io |\psi|^{p-2}\nabla\psi\ft(\gu)\, dx\\
		&\leq &\frac{1}{\tau}\|u-v\|_p\|\psi\|_p^{p-1}+\frac{1}{2}(p-1)\int_\Omega|\psi|^{p-2}|\nabla\psi|^2\, dx\nonumber\\
		&&+\frac{1}{2}(p-1)\int_\Omega|\psi|^{p-2}|\ft(\gu)|^2\, 
		dx,
	\end{eqnarray*}
	from whence follows
	\begin{equation}
	\|\psi\|_p\leq c\|u\|_p +c.\label{at11}
	\end{equation}
	Here $c$ depends on $\tau, p, \f, v$, but not $\sigma$.
	Next we claim that 
	\begin{equation}
	\|\psi\|_2\leq c\|v\|_2+c.\label{at22}
	\end{equation}
	Upon using $\psi$ as a test function in (\ref{ot9}), we obtain
	\begin{equation}
	\int_\Omega|\nabla\psi|^2\, dx
	+\tau\int_\Omega\psi^2\, dx
	=-\frac{\sigma}{\tau}\int_\Omega u\psi \, dx
	+\frac{\sigma}{\tau}\int_\Omega v\psi \, dx-\sigma\io\nabla\psi\ft(\gu)\, dx.\label{ot11}
	\end{equation}
	By using $u$ as a test function in (\ref{ot10}), we can derive
	\begin{equation}\label{at44}
	\int_\Omega u\psi \, dx
	=\int_\Omega|\nabla u|^2\, dx
	+\tau\int_\Omega u^2 \, dx\geq 0.
	\end{equation}
	Keeping this in mind, we deduce from (\ref{ot11}) that
	$$\tau\int_\Omega\psi^2 \, dx
	\leq \frac{\sigma}{\tau}\int_\Omega v\psi \, dx+\frac{1}{2}\io |\ft(\gu)|^2\, dx
	\leq \frac{1}{\tau}\|v\|_2\|\psi\|_2+c.$$
	Then \eqref{at22} follows. 
	
	We can also infer from the classical $L^\infty$ estimate (\cite{GT}, p.189) that for each $p>\frac{N}{2}$ there is a positive
	number $c$ such that
	\begin{equation}\label{at33}
	\|\psi\|_\infty\leq c\|u-v\|_p+c\|\ft(\gu)\|_{2p}+c\|\psi\|_2\leq c\|u\|_p+c.
	\end{equation}
	
	Equipped with \eqref{at11}, \eqref{at22}, and \eqref{at33}, we can employ a bootstrap argument to obtain (\ref{ot8}). 
	Square both sides of (\ref{ot10}) and integrate to obtain
	$$\int_\Omega(\Delta u)^2dx+\tau^2\int_\Omega u^2dx+2\tau\int_\Omega|\nabla u|^2dx=\int_\Omega\psi^2dx.$$
	This together with \eqref{cz} and \eqref{at44}
	implies that \begin{equation}
	\|u\|_{W^{2,2}(\Omega)}\leq c.\label{ot13}
	\end{equation}
	If this is enough to yield that
	\begin{equation}
	\|u\|_p\leq c \ \ \textup{for some $p>\frac{N}{2}$},\label{ot14}
	\end{equation}
	then \eqref{ot8} follows from \eqref{at33}.
	If (\ref{ot13}) is not
	enough for (\ref{ot14}) to hold, by the Sobolev Embedding Theorem we must have that $N\geq 8$ and
	$$\|u\|_{\frac{2N}{N-4}}\leq c.$$
	In view of \eqref{cz} and \eqref{at11}, we have
	$$\|u\|_{W^{2,\frac{2N}{N-4}}(\Omega)}\leq c.$$
	Does this imply (\ref{ot14})? If it does, we are done. If not, repeat the above argument.
	Obviously, we can reach (\ref{ot14}) in a finite number of steps.
	%
	
	Finally, we can conclude from the classical regularity theory that $\psi\in W^{1,p}(\Omega)$ for each $p>1$.
	The proof is complete.
\end{proof}

For a similar approach we refer the reader to Proposition 2.1 in \cite{LX}.
\section{Proof of Theorem \ref{mthm}.}\label{sec5}
The proof of Theorem \ref{mthm} will be divided into several steps. To begin, we form a sequence
of approximate solutions. This is based upon Proposition \ref{app}. Then
we proceed to derive estimates similar to the ones in Proposition \ref{estim} for the sequence. It turns out that 
%
these estimates are 
sufficient to justify passing to the limit.

Let $T>0$ be given. For each $j\in\{1,2,\cdots,\}$ we divide the time interval $[0,T]$ into $j$ equal subintervals. Set
$$\tau=\frac{T}{j}.$$
We discretize  (\ref{mp1})-(\ref{mp4}) as follows. For $k=1,\cdots, j$, we solve recursively the system
\begin{eqnarray}
\frac{u_k-u_{k-1}}{\tau}-\textup{div}\left(\nabla\sk+\ft(\nabla\uk)\right)+\tau\sk&=&0\ \ \ \textup{in $\Omega$},\label{s31}\\
-\Delta\uk+\tau\uk&=&\sk\ \ \ \textup{in $\Omega$},\label{s32}\\
\left(\nabla\sk+\ft(\nabla\uk)\right)\cdot\nu=\nabla\uk\cdot\nu&=&0\ \ \ \textup{on $\partial\Omega$}.\label{s33}
\end{eqnarray}
Without loss of generality, we may assume that $u_0\in W^{1,2}(\Omega)\cap L^\infty(\Omega)$. Otherwise, one picks a sequence $\{u_{0j}\}$ from the space with the property
\begin{equation}
u_{0j}\rightarrow u_0\ \ \textup{strongly in $ W^{1,2}(\Omega)$ and weak$^*$ in $L^\infty(\Omega)$.}
\end{equation}
Introduce the functions
\begin{eqnarray}
\utj(x,t)&=&\frac{t-t_{k-1}}{\tau}\uk(x)+(1-\frac{t-t_{k-1}}{\tau})\uko(x), \ x\in\Omega,  \ t\in(t_{k-1},t_k],\\
\ubj(x,t)&=&\uk(x), \ \ \ x\in\Omega, \ \ t\in(t_{k-1},t_k],\\
\sbj(x,t)&=&\skx, \ \ \ x\in\Omega, \ \ t\in(t_{k-1},t_k],
\end{eqnarray}
where $t_k=k\tau$. We can rewrite (\ref{s31})-(\ref{s33}) as
\begin{eqnarray}
\frac{\partial\utj}{\partial t}-\textup{div}\left(\nabla\sbj+\ft(\nabla\ubj)\right)+\tau\sbj&=&0\ \ \ \textup{in $\omt$},\label{omm1}\\
-\Delta\ubj+\tau\ubj &=&\sbj \ \ \ \textup{in $\omt$}.\label{omm2}
\end{eqnarray}
We proceed to derive a priori estimates for the sequence of approximate solutions $\{\utj,\ubj,\sbj\}$. The discretized version of Proposition \ref{estim} is the following
\begin{prop}\label{mdisc}Let (H1)-(H3) be satisfied. The following statements hold true:
	\begin{enumerate}
		\item[\textup{(d1)}] If $\alpha\leq 1$ then for each $T>0$ there is a positive number $c$ independent of $j$, hence $\tau$, such that
		\begin{eqnarray}\label{discm}
		\lefteqn{	\max_{0\leq t\leq T}\left(	\io(\ubj^2+|\nabla\ubj|^2)dx+\tau\io\ubj^2dx\right)}\nonumber\\
		&&+\int_{\Omega_T}(\sbj^2+|\nabla\sbj|^2)dxdt+\tau\int_{\Omega_T}\sbj^2dxdt\leq c;
		\end{eqnarray}
		\item[\textup{(d2)}] If either \eqref{con1} or \eqref{con2} holds, then there is a positive number $T_0$ with
		\begin{eqnarray}\label{discm2}
		\lefteqn{	\max_{0\leq t\leq T_0}\left(	\io(\ubj^2+|\nabla\ubj|^2)dx+\tau\io\ubj^2dx\right)}\nonumber\\
		&&+\int_{\Omega_{T_0}}(\sbj^2+|\nabla\sbj|^2)dxdt+\tau\int_{\Omega_{T_0}}\sbj^2dxdt+\int_{\Omega_{T_0}}|\nabla\ubj|^{2\alpha}dxdt\leq c.
		\end{eqnarray}
		Here $c$ has the same meaning as the one in (d1).
	\end{enumerate}
\end{prop}

\begin{proof}
	Multiply through \eqref{s31} by $\uk$ and integrate the resulting equation over $\Omega$ to obtain
	\begin{equation}\label{disc1}
	\frac{1}{\tau}\io (\uk-u_{k-1})\uk dx+\io\left(\nabla\sk\nabla\uk+\tau\sk\uk\right)dx=-\io\ft(\nabla\uk)\nabla\uk dx.
	\end{equation}
	Similarly, multiply each term in \eqref{s32} by $\sk$ and integrate to get
	\begin{equation}
	\io\left(\nabla\sk\nabla\uk+\tau\sk\uk\right)dx=\io\psi_k^2dx.
	\end{equation}
	Use this, (H1), and the inequality $(\uk-u_{k-1})\uk\geq\frac{1}{2}(\uk^2-u_{k-1}^2)$ in \eqref{disc1} to derive
	\begin{equation}
	\frac{1}{2\tau}\io(\uk^2-u_{k-1}^2)dx+\io\psi_k^2dx\leq c\io|\nabla\uk|^{\alpha+1}dx +c.\label{disc2}
	\end{equation}
	Upon using $\sk$ as a test function in \eqref{s31}, we arrive at
	\begin{eqnarray}
	\lefteqn{\frac{1}{\tau}\io(\uk-u_{k-1})\sk dx+\io|\nabla\sk|^2dx+\tau\io\psi_k^2 dx}\nonumber\\
	&=&-\io\ft(\nabla\uk)\nabla\sk dx\nonumber\\
	&\leq&\frac{1}{2}\io|\nabla\sk|^2dx+\frac{1}{2}\io|\ft(\nabla\uk)|^2dx\nonumber\\
	&\leq&\frac{1}{2}\io|\nabla\sk|^2dx+c\io|\nabla\uk|^{2\alpha}dx+c.\label{disc3}
	\end{eqnarray}
	We calculate from \eqref{s32} that
	\begin{eqnarray}
	\frac{1}{\tau}\io(\uk-u_{k-1})\sk dx&=&\frac{1}{\tau}\io(\uk-u_{k-1})(-\Delta\uk+\tau\uk) dx\nonumber\\
	&=&\frac{1}{\tau}\io(\nabla\uk-\nabla u_{k-1})\nabla\uk dx+\io(\uk-u_{k-1})\uk dx\nonumber\\
	&\geq&\frac{1}{2\tau}\io(|\nabla\uk|^2-|\nabla u_{k-1}|^2)dx+\frac{1}{2}\io(\uk^2-u_{k-1}^2)dx.
	\end{eqnarray}
	Combining this with \eqref{disc3} yields
	\begin{eqnarray}
	\lefteqn{\frac{1}{\tau}\io(|\nabla\uk|^2-|\nabla u_{k-1}|^2)dx+\io(\uk^2-u_{k-1}^2)dx}\nonumber\\
	&&+\io|\nabla\sk|^2dx+\tau\io\psi_k^2 dx
	\leq c\io|\nabla\uk|^{2\alpha}dx+c.\label{disc4}
	\end{eqnarray}
	The case where $\alpha\leq 1$ is easy to handle. From now on, we will operate under the assumption \eqref{con1}. Add \eqref{disc4} to \eqref{disc2} to obtain
	\begin{eqnarray}
	\lefteqn{\frac{1}{\tau}\io(u_k^2-u_{k-1}^2)dx+\frac{1}{\tau}\io(|\nabla\uk|^2-|\nabla u_{k-1}|^2)dx}\nonumber\\
	&&+\io(\uk^2-u_{k-1}^2)dx+\io\psi_k^2dx+\io|\nabla\sk|^2dx+\tau\io\psi_k^2 dx\nonumber\\
	&	\leq &c\io|\nabla\uk|^{2\alpha}dx+c.\label{disc5}
	\end{eqnarray}
	By \eqref{gue}, for each $\varepsilon>0$ there is a positive number
	$c(\varepsilon)$ such that
	\begin{eqnarray}
	\io|\nabla\uk|^{2\alpha}dx&\leq& \varepsilon\io\left(\left(\nabla\Delta\uk\right)^2+\left(\Delta\uk\right)^2\right)dx+c(\varepsilon)\left(\io|\nabla\uk|^2dx\right)^{1+\sigma}\nonumber\\
	&\leq& \varepsilon\io\left(\left(\nabla(\tau\uk-\sk)\right)^2+\left(\tau\uk-\sk\right)^2\right)dx+c(\varepsilon)\left(\io|\nabla\uk|^2dx\right)^{1+\sigma}\nonumber\\
	&\leq& 2\varepsilon\io\left(|\nabla\sk|^2+\psi_k^2\right)dx+2\varepsilon\tau^2\io\left(|\nabla\uk|^2+u_k^2\right)dx\nonumber\\
	&&+c(\varepsilon)\left(\io|\nabla\uk|^2dx\right)^{1+\sigma},\label{disc10}
	\end{eqnarray}
	where $\sigma$ is a positive number determined by $N, \alpha$. Use this in \eqref{disc5} and choose $\varepsilon$ suitably small to derive
	\begin{eqnarray}
	\lefteqn{\frac{1}{\tau}\io(u_k^2-u_{k-1}^2)dx+\frac{1}{\tau}\io(|\nabla\uk|^2-|\nabla u_{k-1}|^2)dx}\nonumber\\
	&&+\io(\uk^2-u_{k-1}^2)dx+\io\psi_k^2dx+\io|\nabla\sk|^2dx+\tau\io\psi_k^2 dx\nonumber\\
	&	\leq &c\varepsilon\tau^2\io\left(|\nabla\uk|^2+u_k^2\right)dx
	+c(\varepsilon)\left(\io|\nabla\uk|^2dx\right)^{1+\sigma}+c.
	\end{eqnarray}
	Multiply through this inequality by $\tau$ and sum up the resulting inequality over $k$ to derive
	\begin{eqnarray}
	\lefteqn{\io(\ubj^2+|\nabla\ubj|^2)dx+\tau\io\ubj^2dx+\int_{\Omega_s}(\sbj^2+|\nabla\sbj|^2)dxdt+\tau\int_{\Omega_s}\sbj^2dxdt}\nonumber\\
	&\leq&c\tau^2\int_{\Omega_s}(\ubj^2+\nabla\ubj|^2)dxdt+c\int_{0}^{s}\left(\io|\nabla\ubj|^2dx\right)^{1+\sigma}dt+c\nonumber\\
	&\leq&c\int_{0}^{s}\left(\io(\ubj^2+|\nabla\ubj|^2)dx\right)^{1+\sigma}dt+c,
	\end{eqnarray}
	where $s\geq 0$.	Set 
	\begin{equation}
	y_j(s)=\int_{0}^{s}\left(\io(\ubj^2+|\nabla\ubj|^2)dx\right)^{1+\sigma}dt.
	\end{equation}
	We easily infer from the preceding inequality that
	\begin{equation}
	y_j^\prime(s)\leq cy_j^{1+\sigma}(s)+c.
	\end{equation}
	Thus we are in a position to apply Lemma \ref{gron} to obtain
	\begin{equation}
	y_j(s)\leq \frac{1}{\left[\left(((y_j(0)+1)^{-\sigma}+1)e^{-c\sigma t}-1\right)^+\right]^{\frac{1}{\sigma}}}-1.
	\end{equation}
	As before, we can conclude from this, \eqref{disc10}, and \eqref{disc5} that there exists a positive $T_0$ such that \eqref{discm} is true. The case where \eqref{con2} holds can be handled
	in a similar way. The proof is complete.
\end{proof}

{\it Proof of Theorem \ref{mthm}.} 
Here we draw some inspiration from \cite{X}. We only deal with the case where \eqref{con1} holds because the other two cases are similar. Proposition \ref{mdisc} together with \eqref{omm1}, \eqref{cz} and \eqref{omm2}
implies
\begin{enumerate}
	\item[(C1)] the sequence $\{\ubj\}$ is bounded in $L^\infty(0,T_0;W^{1,2}(\Omega))$;
	\item[(C2)]
	the sequence $\{\ubj\}$ is bounded in $ L^2(0,T_0;W^{2,2}(\Omega))$;
	\item[(C3)]
	the sequence $\{\Delta\ubj\}$ is bounded in $ L^2(0,T_0;W^{1,2}(\Omega))$; and
	\item[(C4)]
	the sequence  $\{|\f(\nabla\uk)|\}$ is bounded in $ L^2(\Omega_{T_0})$.
\end{enumerate}

Now we proceed to show that the sequence $\{\ubj\}$ is precompact in $L^2(0,T_0;W^{1,2}(\Omega))$. It is easy to see from (C1) that the sequence $\{\utj\}$ is bounded in $L^\infty(0,T_0;W^{1,2}(\Omega))$. We can also conclude from (C4), Proposition \ref{mdisc}, and \eqref{omm1} that the sequence $\{\partial_t\utj\}$ is bounded in $L^2(0,T_0;(W^{1,2}(\Omega))^*)$, where $(W^{1,2}(\Omega))^*$ is the dual space of $W^{1,2}(\Omega)$.	This puts us in a position to invoke a result in \cite{S}, from which follows that $\{\utj\}$ is precompact in both $L^2(0,T_0;(W^{1,2}(\Omega))^*)$ and $C([0,T_0]; L^2(\Omega))$.
For $t\in(t_{k-1}, t_k]$, we calculate from \eqref{omm1} that
\begin{eqnarray*}
	\utj(x,t)-\ubj(x,t)&=&\frac{t-t_k}{\tau}(\uk-\uko)\\
	&=&(t-t_k)\frac{\partial}{\partial t}\utj(x,t)\nonumber\\
	&=&(t-t_k)\left(\textup{div}\left(\nabla\sk+\f(\nabla\uk)\right)-\tau\sk\right).
\end{eqnarray*}
Consequently, we can estimate
\begin{eqnarray}
\int_{0}^{T_0}\|\utj-\ubj\|^2_{(W^{1,2}(\Omega))^*}dt&=&\sum_{k=1}^{k=j}\int_{t_{k-1}}^{t_k}\|\utj-\ubj\|^2_{(W^{1,2}(\Omega))^*}dt\nonumber\\
&=&\sum_{k=1}^{k=j}\int_{t_{k-1}}^{t_k}(t_k-t)^2\|\textup{div}\left(\nabla\sk+\f(\nabla\uk)\right)-\tau\sk\|^2_{(W^{1,2}(\Omega))^*}dt\nonumber\\
&=&\sum_{k=1}^{k=j}\tau^3\|\textup{div}\left(\nabla\sk+\f(\nabla\uk)\right)-\tau\sk\|^2_{(W^{1,2}(\Omega))^*}\nonumber\\	&=&\tau^2\int_{0}^{T_0}\|\textup{div}\left(\nabla\sbj+\f(\nabla\ubj)\right)-\tau\sbj\|^2_{(W^{1,2}(\Omega))^*}dt\nonumber\\
&\leq& c\tau^2,\label{fin1}
\end{eqnarray}
from whence follows that $\{\ubj\}$ is precompact in $L^2(0,T_0;(W^{1,2}(\Omega))^*)$. This together with (C1) and a
result in \cite{S} asserts that $\{\ubj\}$ is precompact in $L^2(\ot)$.
We compute
\begin{eqnarray}
\int_{\Omega_{T_0}}|\nabla(\rbj-\bar{u}_i)|^2\, dxdt
&=&\int_{\Omega_{T_0}}(\Delta \rbj-\Delta\bar{u}_i )(\rbj-\bar{u}_i)\, dxdt\nonumber\\
&\leq& c\left(\int_{\Omega_{T_0}}(\rbj-\bar{u}_i)^2\, dxdt\right)^{\frac{1}{2}}
\end{eqnarray}
for all $i,j$.  Hence $\{\ubj\}$ is precompact in $L^2(0,T_0; W^{1,2}(\Omega))$.

Select a subsequence of $j$, still denoted by $j$, such that 
\begin{eqnarray}
\ubj &\rightarrow & u\ \ \textup{strongly in $L^2(0,T_0; W^{1,2}(\Omega))$},\\
\nabla\ubj &\rightarrow & \nabla u\ \ \textup{a.e. on $\ot$},\label{fin2}\\
\utj &\rightarrow & \tilde{u}\ \ \textup{strongly in $C([0,T_0]; L^2(\Omega))$}.
\end{eqnarray}
Since $\f$ is continuous, (C4) combined with \eqref{fin2} implies
\begin{equation*}
\f(\nabla\ubj)\rightharpoonup\f(\nabla u)\ \ \textup{weakly in $L^2(0,T_0; (L^2(\Omega))^N)$}.
\end{equation*}
We can deduce from \eqref{fin1} that
\begin{equation*}
u=\tilde{u} \ \ \textup{a.e. on $\Omega_{T_0}$.}
\end{equation*}
We are ready to pass to the limit in \eqref{omm1} and \eqref{omm2}, from which follows the theorem. The proof
is finished.
\section{Proof of Theorem \ref{mthm2}.} \label{sec6}
We begin with the proof of (R3) in Theorem \ref{mthm2}.

\noindent {\it Proof of (R3). }
Let $\theta(s)$ be a function in $C_0^\infty(\R)$ satisfying
\begin{equation}
\theta(s)=s\ \ \textup{on $(-1, 1)$.}
\end{equation}
Then denote by $\h(\nabla u)$ the function $(g_1(\nabla u_0)+\theta(g_1(\gu)-g_1(\nabla u_0)),\cdots, g_N(\nabla u_0)+\theta(g_N(\gu)-g_N(\nabla u_0)))^T$. Note that
\begin{equation}\label{the}
\h(\nabla u)=\f(\nabla u)\ \ \textup{on the set where $|\f(\gu)-\f(\nabla u_0)|\leq 1$.}
\end{equation}	On the other hand, since $\theta$ is bounded and $\f$ is continuous, there is a positive number $M$ with
\begin{equation}\label{hb}
\sup_{\rnt}|\h(\gu)|\leq M\ \ \textup{for each $u\in L^1(0,\infty; W^{1,1}(\RN))$}.
\end{equation} 	
Consider the initial value problem
\begin{eqnarray}
\partial_tu+\dus&=&\textup{div}\h(\nabla u)\ \ \textup{in $\RN\times(0,\infty)$,}\label{agp1}	\\
u(x,0)&=&u_0(x)\ \ \textup{on $\RN$.}\label{agp2}
\end{eqnarray}	
Without loss of generality, we may assume that $\f, u_0$ are sufficiently smooth. Otherwise, we replace them by their respective mollifications and then pass to the limit.
We conclude from a standard argument (see, e.g., \cite{GP} or (\cite{T}, Chap. 15)) that \eqref{agp1}-\eqref{agp2} has a classical solution $u$ and that
$u$ can be represented in the form
\begin{equation}
u=v_0+v_1,
\end{equation} 	
where \begin{eqnarray*}
	v_0(x,t)&=&b_N( t)*u_0(x)\equiv\alpha_Nt^{-\frac{N}{4}}\irn\fn\left(\frac{|x-y|}{t^{\frac{1}{4}}}\right)u_0(y)dy,\\
	v_1(x,t)&=&\int_{0}^{t}\nabla b_N( t-s)*\h(\gu)(x)ds\nonumber\\
	&&\equiv\alpha_N\int_{0}^{t}(t-s)^{-\frac{N+1}{4}}\irn\fn^\prime\left(\frac{|x-y|}{(t-s)^{\frac{1}{4}}}\right)\frac{x-y}{|x-y|}\h(\gu(y, s))dyds.
\end{eqnarray*}
Whenever no confuse arises, we suppress the dependence of a function on its independent variables. 
\begin{clm}\label{clm1}If $u_0\in C^1(\RN)$ with $\nabla u_0$ being uniformly continuous and bounded on $\RN$, then $\lim_{t\rightarrow 0}|\nabla v_0(x,t)-\nabla u_0(x)|=0$ uniformly on $\RN$.
\end{clm}
\begin{proof} 
	We easily derive from the definition of $v_0$ and \eqref{norm} that
	\begin{eqnarray}
	\nabla v_0(x,t)&=&\alpha_Nt^{-\frac{N}{4}}\irn\fn\left(\frac{|x-y|}{t^{\frac{1}{4}}}\right)\nabla u_0(y)dy\nonumber\\
	&=&\alpha_Nt^{-\frac{N}{4}}\irn\fn\left(\frac{|x-y|}{t^{\frac{1}{4}}}\right)(\nabla u_0(y)-\nabla u_0(x))dy+\nabla u_0(x)\nonumber\\
	&=&\alpha_N\irn\fn\left(|z|\right)(\nabla u_0(x-t^{\frac{1}{4}}z)-\nabla u_0(x))dz+\nabla u_0(x).\label{6.9}
	\end{eqnarray}
	The claim follows
	from (F2) and the boundedness and uniform continuity of $ \nabla u_0$ on $\RN$.
\end{proof}

In view of the calculations in (\cite{LSU}, p. 263-264), we can compute
\begin{eqnarray}
\nabla v_1
&=&\alpha_N\int_{0}^{t}(t-s)^{-\frac{N+2}{4}}\irn\fn^{\prime\prime}\left(\frac{|x-y|}{(t-s)^{\frac{1}{4}}}\right)\frac{(x-y)\otimes(x-y)}{|x-y|^2}\h(\gu(y, s))dyds\nonumber\\
&&+\alpha_N\int_{0}^{t}(t-s)^{-\frac{N+2}{4}}\irn\fn^{\prime}\left(\frac{|x-y|}{(t-s)^{\frac{1}{4}}}\right)\nabla\frac{x-y}{|x-y|}\h(\gu(y, s))dyds
\end{eqnarray}
We can easily check that\begin{equation}
\left|\nabla\frac{x-y}{|x-y|}\right|\leq \frac{N^2}{|x-y|}.
\end{equation}
Keeping this, (F3), \eqref{stan2}, and \eqref{hb} in mind, we arrive at	
\begin{eqnarray}
|\nabla v_1|
&\leq&c\int_{0}^{t}(t-s)^{-\frac{N+2}{4}}\irn\left|\fn^{\prime\prime}\left(\frac{|x-y|}{(t-s)^{\frac{1}{4}}}\right)\right|dyds\nonumber\\
&&+c\int_{0}^{t}(t-s)^{-\frac{N+2}{4}}\irn\left|f_{N+2}\left(\frac{|x-y|}{(t-s)^{\frac{1}{4}}}\right)\right|
dyds\nonumber\\
&\leq&c\int_{0}^{t}(t-s)^{-\frac{1}{2}}ds= ct^{\frac{1}{2}}.\label{loc2}
\end{eqnarray}
With the aid of Claim \ref{clm1}, \eqref{loc2}, and the continuity of $\f$, we can find a positive number $T$ such that
\begin{equation}
|\f(\gu)-\f(\nabla u_0)|\leq 1\ \ \textup{on $\RN\times(0,T)$}.
\end{equation}
This together with \eqref{the} implies
\begin{equation}
\partial_tu+\dus=\textup{div}\f(\nabla u)\ \ \textup{in $\RN\times(0,T)$.}
\end{equation}
This completes the proof of (R3).

Before we prove (R4), we present some preliminaries. For this purpose we construct a sequence $\{\wk\}$ by successively solving
\begin{eqnarray}
\partial_t\wk+\ds\wk&=&\textup{div}\left(\f(\nabla\wko)\right)\  \  \textup{in $\RN\times(0,\infty)$,}\label{agp11}\\
\wk(x,0)&=& u_0(x)\   \  \textup{on $\RN,\  \  k=1, 2, \cdots$,}\label{agp12}
\end{eqnarray}where
\begin{equation}\label{wvo}
w_0=v_0.
\end{equation} 
As before, we have
\begin{eqnarray}
\wk(x,t)&=&\alpha_N\int_{0}^{t}(t-s)^{-\frac{N+1}{4}}\int_{\RN}\fn^\prime\left(\frac{|x-y|}{(t-s)^{\frac{1}{4}}}\right)\frac{x-y}{|x-y|}\f(\nabla\wko(y,s))dyds\nonumber\\
&&+v_0(x,t).
\end{eqnarray}
\begin{lem}
	For each $p\geq 1$ there is a positive number $c=c(p, N)$ such that
	\begin{equation}\label{lem11}
	\|\nabla v_0(\cdot, t)\|_\infty\leq ct^{-\frac{N+p}{4p}}\|u_0\|_p \ \ \textup{for $t>0$}.
	\end{equation}
	The above inequality is still valid if $p=\infty$. In this case, $c$ only depends on $N$.
\end{lem}
\begin{proof} Let $p>1$ be given and set $q=\frac{p}{p-1}$.
	By Young's inequality for convolutions, we have
	\begin{eqnarray}
	|\nabla v_0(x,t)|&=& \alpha_Nt^{-\frac{N+1}{4}}\left|\irn\fn^\prime\left(\frac{|x-y|}{t^{\frac{1}{4}}}\right)\frac{x-y}{|x-y|}u_0(y)dy\right|\nonumber\\
	&\leq&\alpha_Nt^{-\frac{N+1}{4}}\|u_0\|_p\left(\irn\left|\fn^\prime\left(\frac{|y|}{t^{\frac{1}{4}}}\right)\right|^qdy\right)^{\frac{1}{q}}\nonumber\\
	&\leq &ct^{-\frac{N+1}{4}+\frac{N}{4q}}\|u_0\|_p.\label{u1e2}
	\end{eqnarray}
	The last step is due to \eqref{stan2}.
	If $p=1$, the above calculations still work.
	This completes the proof.
\end{proof}

\noindent{\it Proof of (R4) in Theorem \ref{mthm2}.}
Fix $T>0$. Let
\begin{equation}
b_k=\max_{\RN\times [0,T]}|t^{\frac{1}{2(\alpha-1)}}\nabla\wk(x,t)|,\  \   k=0, 1, 2, \cdots.
\end{equation}
By a calculation similar to \eqref{loc2}, we have , for $(x,t)\in \RN\times [0,T]$, that
\begin{eqnarray}
|\nabla w_k|&\leq&c\int_{0}^{t}(t-s)^{-\frac{N+2}{4}}\int_{\RN}\left|\fn^{\prime\prime}\left(\frac{|x-y|}{(t-s)^{\frac{1}{4}}}\right)\right||\nabla\wko|^\alpha dyds\nonumber\\
&&+c\int_{0}^{t}(t-s)^{-\frac{N+2}{4}}\int_{\RN}\left|f_{N+2}\left(\frac{|x-y|}{(t-s)^{\frac{1}{4}}}\right)\right||\nabla\wko|^\alpha 
dyds+|\nabla v_0|\nonumber\\
&\leq&cb_{k-1}^\alpha\int_{0}^{t}(t-s)^{-\frac{N+2}{4}}s^{-\frac{\alpha}{2(\alpha-1)}}\int_{\RN}\left|\fn^{\prime\prime}\left(\frac{|x-y|}{(t-s)^{\frac{1}{4}}}\right)\right|dyds\nonumber\\
&&+cb_{k-1}^\alpha\int_{0}^{t}(t-s)^{-\frac{N+2}{4}}s^{-\frac{\alpha}{2(\alpha-1)}}\int_{\RN}\left|f_{N+2}\left(\frac{|x-y|}{(t-s)^{\frac{1}{4}}}\right)\right|
dyds+|\nabla v_0|\nonumber\\
&\leq&cb_{k-1}^\alpha\int_{0}^{t}(t-s)^{-\frac{1}{2}}s^{-\frac{\alpha}{2(\alpha-1)}}ds+|\nabla v_0|.\label{small2}
\end{eqnarray}
Multiply through this inequality by $t^{\frac{1}{2(\alpha-1)}} $ and note from our assumptions on $\alpha$ that
\begin{equation}
t^{\frac{1}{2(\alpha-1)}}\int_{0}^{t}(t-s)^{-\frac{1}{2}}s^{-\frac{\alpha}{2(\alpha-1)}}ds\leq c \ \ \textup{for all $t>0$}
\end{equation}
to derive
\begin{equation}
|t^{\frac{1}{2(\alpha-1)}} \nabla\wk(x,t)|\leq 	|t^{\frac{1}{2(\alpha-1)}} \nabla v_0(x,t)|+cb_{k-1}^\alpha=b_0+cb_{k-1}^\alpha\  \  \textup{for $(x,t)\in \RN\times [0,T]$.}
\end{equation}
The last step is due to \eqref{wvo}.
Taking $p=\frac{(\alpha-1)N}{3-\alpha}$ in \eqref{lem11}, we deduce 
\begin{equation}
b_0=|t^{\frac{1}{2(\alpha-1)}} \nabla v_0(x,t)|\leq  c\|u_0\|_{\frac{(\alpha-1)N}{3-\alpha}}.
\end{equation}
By Lemma \ref{small}, we have 
\begin{equation}
\max_{\RN\times [0,T]}|t^{\frac{1}{2(\alpha-1)}}\nabla\wk(x,t)|=b_k\leq c,\label{small4}
\end{equation}
provided that $\|u_0\|_{\frac{(\alpha-1)N}{3-\alpha}} $ is sufficiently small. From here on we assume that this is the case. Note that  the preceding inequality holds for $T=\infty$ because the constant $c$ on its right-hand side  does not depend on $T$. Recall
\begin{equation}
\nabla v_0(x,t)=\alpha_Nt^{-\frac{N}{4}}\irn\fn\left(\frac{|x-y|}{t^{\frac{1}{4}}}\right)\nabla u_0(y)dy,
\end{equation}
from whence follows
\begin{equation}
\max_{\RN}|	\nabla v_0(x,t)|\leq c\|\nabla u_0\|_\infty\ \   \textup{for each $t>0$.}\label{small3}
\end{equation}
For $\tau>0$ we set 
\begin{equation}
a_k(\tau )=\max_{\RN\times [0,\tau]}|\nabla\wk(x,t)|,\ \ \ k=0, 1, 2, \cdots.
\end{equation}
We infer from \eqref{loc2}, \eqref{wvo}, and \eqref{small2} that
\begin{equation}
a_k(\tau )\leq a_0(\tau)+c\tau^{\frac{1}{2}}	a_{k-1}^\alpha(\tau ), \ \ k=1,2,\cdots.
\end{equation}
It follows from from \eqref{small3} that 
\begin{equation}
a_0(\tau)\leq c\|\nabla u_0\|_\infty.
\end{equation}Invoking Lemma \ref{small} again, we see that for $\tau$ small enough there holds
\begin{equation}
a_k(\tau )\leq c \ \ \textup{for all $k$}.
\end{equation}
This combined with \eqref{small4} asserts that 
\begin{equation}
\max_{\RN\times [0,T]}|\nabla\wk(x,t)|\leq c.
\end{equation}
Observe that the constant $c$ here is independent of $T$. We actually have
\begin{equation}
\max_{\RN\times [0,\infty)}|\nabla\wk(x,t)|\leq c.
\end{equation}
Since $\g$ is locally Lipschitz, we can find a positive number $c$ such that
\begin{equation}
|\g(\nabla\wk)-\g(\nabla\wko)|\leq c|\nabla\wk-\nabla\wko|\ \ \textup{for each $k\geq 1$.}
\end{equation}
Set
\begin{equation}
d_k=\max_{\RN\times [0,T]}|\nabla\wk(x,t)-\nabla\wko(x,t)|.
\end{equation}
For $k=2,3, \cdots$ and $(x,t)\in \RN\times [0,T]$, we compute
\begin{eqnarray}
|\nabla w_k-\nabla w_{k-1}|&\leq&c\int_{0}^{t}(t-s)^{-\frac{N+2}{4}}\int_{\RN}\left|\fn^{\prime\prime}\left(\frac{|x-y|}{(t-s)^{\frac{1}{4}}}\right)\right||\nabla w_{k-1}-\nabla w_{k-2}| dyds\nonumber\\
&&+c\int_{0}^{t}(t-s)^{-\frac{N+2}{4}}\int_{\RN}\left|f_{N+2}\left(\frac{|x-y|}{(t-s)^{\frac{1}{4}}}\right)\right||\nabla w_{k-1}-\nabla w_{k-2}|
dyds\nonumber\\
&\leq&cd_{k-1}\int_{0}^{t}(t-s)^{-\frac{N+2}{4}}\int_{\RN}\left|\fn^{\prime\prime}\left(\frac{|x-y|}{(t-s)^{\frac{1}{4}}}\right)\right|dyds\nonumber\\
&&+cd_{k-1}\int_{0}^{t}(t-s)^{-\frac{N+2}{4}}\int_{\RN}\left|f_{N+2}\left(\frac{|x-y|}{(t-s)^{\frac{1}{4}}}\right)\right|
dyds\nonumber\\
&\leq&cd_{k-1}\int_{0}^{t}(t-s)^{-\frac{1}{2}}ds=cd_{k-1}t^{\frac{1}{2}}.\label{conv1}
\end{eqnarray}
That is, we have
\begin{equation}
d_k\leq cT^{\frac{1}{2}}d_{k-1}\leq \cdots\leq \left( cT^{\frac{1}{2}}\right)^{k-1}d_1.
\end{equation}
Choose $T$ so that
\begin{equation}\label{cont}
cT^{\frac{1}{2}}<1.
\end{equation}Then the series 
\begin{equation}
\nabla w_0+\nabla w_1-\nabla w_0+\nabla w_2-\nabla w_1+\cdots+\nabla\wk-\nabla\wko+\cdots
\end{equation}
is uniformly convergent on $\RN\times [0,T]$. Thus the whole sequence $\{\nabla\wk\}$ converges strongly in $C(\RN\times [0,T])$.
By a simple translation in the $t$-variable, we can show that  for each $t_0>0$ the sequence $\{\nabla\wk\}$ converges strongly in $C(\RN\times [t_0,t_0+T])$, where $T$ is given as in \eqref{cont}, as long as  $\{\nabla\wk(x,t_0)\}$ converges strongly in
$C(\RN)$. This immediately implies that $\{\nabla\wk\}$ converges strongly in $C(\RN\times [0,s])$ for each $s>0$. Thus we can take
$k\rightarrow\infty$ in \eqref{agp11} to obtain the desired result. This completes the proof of (R4).



\end{document}